\documentclass[a4paper,11pt]{article}
\usepackage{amssymb}
\usepackage{amsthm}
\usepackage{amsmath}
\usepackage{enumerate}

\newtheorem{theorem}{Theorem}[section]
\newtheorem{lemma}[theorem]{Lemma}
\newtheorem{corollary}[theorem]{Corollary}
\theoremstyle{definition}
\newtheorem{example}[theorem]{Example}

\newcommand{\email}[1]{\texttt{#1}}
\newcommand{\set}[2]{\left\{ #1 \mid #2 \right\}}

\newcommand{\dd}{\cdots}
\newcommand{\Z}{\mathbb Z}
\newcommand{\Q}{\mathbb Q}
\newcommand{\No}{\mathbb N_0}
\newcommand{\Ni}{\mathbb N_1}
\newcommand{\eps}{\varepsilon}

\newcommand{\lcm}{\mathrm{lcm}}
\newcommand{\px}[1]{ P_{#1} }
\newcommand{\rx}[1]{ R_{#1} }
\newcommand{\qx}[3]{ Q_{#1,#2,#3} }
\newcommand{\sx}[2]{ S_{#1,#2} }
\newcommand{\len}[1]{\mathrm{len}_{#1}}
\newcommand{\lhp}{\mathcal M}
\newcommand{\sol}[2]{\mathrm{Sol}_{#1}^{#2}}

\begin{document}

\title{Systems of Word Equations, Polynomials and Linear Algebra: A New Approach
\thanks{Supported by the Academy of Finland under grant 257857}
}

\author{Aleksi Saarela \\
Department of Mathematics and Statistics \\
University of Turku, FI-20014 Turku, Finland \\
\email{amsaar@utu.fi}
}

\maketitle

\begin{abstract}
We develop a new tool,
namely polynomial and linear algebraic methods,
for studying systems of word equations.
We illustrate its usefulness
by giving essentially simpler proofs of several hard problems.
At the same time we prove extensions of these results.
Finally, we obtain the first nontrivial upper bounds
for the fundamental problem of
the maximal size of independent systems.
These bounds depend quadratically on the size of the shortest equation.
No methods of having such bounds have been known before.
\end{abstract}

\section{Introduction}

Combinatorics on words is a part of discrete mathematics.
It studies the properties of strings of symbols and
has applications in many areas from pure mathematics to computer science.
See, e.g., \cite{Lo83} or \cite{ChKa97}
for a general reference on this subject.

Some of the most fundamental questions in combinatorics on words
concern word equations.
First such question is the complexity of the satisfiability problem,
i.e., the problem of determining
whether a given equation with constants has a solution.
The satisfiability problem was proved to be decidable by Makanin \cite{Ma77} and
proved to be in PSPACE by Plandowski \cite{Pl04},
and it has been conjectured to be NP-complete.

A second question is
how to represent all solutions of a constant-free equation.
Hmelevskii proved that the solutions of an equation on three unknowns
can be represented with parametric words,
but this does not hold for four unknowns \cite{Hm71}.
The original proof has been simplified \cite{KaSa08dlt}
and used to study a special case of the satisfiability problem \cite{Sa09dlt}.

A third fundamental question, which is very important for this article,
is the maximal size of an independent system of word equations.
It was proved by Albert and Lawrence \cite{AlLa85ehrenfeucht}
and independently by Guba \cite{Gu86} that
an independent system cannot be infinite.
However, it is still not known
whether there are unboundedly large independent systems.

One of the basic results in the theory of word equations
is that a nontrivial equation causes a defect effect. In other
words, if $n$ words satisfy a nontrivial relation, then they can be
represented as products of $n-1$ words. Not much is known about the
additional restrictions caused by several independent relations
\cite{HaKa04}.

In fact, even the following simple question,
formulated already in \cite{CuKa83},
is still unanswered:
How large can an independent system of word equations on three unknowns be?
The largest known examples consist of three equations.
This question can be obviously asked also in the case of $n > 3$ unknowns.
Then there are independent systems of size $\Theta(n^4)$ \cite{KaPl94}.
Some results concerning independent systems on three unknowns
can be found in \cite{HaNo03}, \cite{CzKa07} and \cite{CzPl09},
but the open problem seems to be
very difficult to approach with current techniques.

There are many variations of the above question: We may study it in
the free semigroup, i.e., require that $h(x) \ne \eps$ for every
solution $h$ and unknown $x$, or examine only the systems having a
solution of rank $n-1$, or study chains of solution sets instead of
independent systems. See, e.g., \cite{HaKaPl02}, \cite{HaKa04},
\cite{Cz08} and \cite{KaSa11adian}.

In this article we will use polynomials to study some
questions related to systems of word equations. Algebraic techniques
have been used before, most notably in the proof of Ehrenfeucht's
conjecture, which is based on Hilbert's basis theorem. However, the
way in which we use polynomials is quite different and allows us to
apply linear algebra to the problems.

The main contribution of this article
is the development of new methods for attacking problems on word equations.
This is done in Sections \ref{sect:fixedlength} and \ref{sect:solsets}.
Other contributions include
simplified proofs and generalizations for old results
in Sections \ref{sect:appl} and \ref{sect:indsyst},
and studying maximal sizes of independent systems of equations
in Section \ref{sect:indsyst}.
Thus the connection between word equations and linear algebra
is not only theoretically interesting,
but is also shown to be very useful at establishing simple-looking results
that have been previously unknown,
or that have had only very complicated proofs.
In addition to the results of the paper,
we believe that the techniques may be useful
in further analysis of word equations.

Next we give a brief overview of the paper. First, in Section
\ref{sect:basic} we define a way to transform words into polynomials
and prove some basic results using these polynomials.

In Section \ref{sect:fixedlength} we prove that if the lengths of the
unknowns are fixed, then there is a connection between the ranks of
solutions of a system of equations and the rank of a certain
polynomial matrix. This theorem is very important for all the later
results.

Section \ref{sect:appl} contains small generalizations of two
earlier results. These are nice examples of the methods developed in
Section \ref{sect:fixedlength} and have independent interest, but they
are not important for the later sections.

In Section \ref{sect:solsets}
we analyze the results of Section \ref{sect:fixedlength}
when the lengths of the unknowns are not fixed.
For every solution these lengths form an $n$-dimensional vector,
called the \emph{length type} of the solution.
We prove that the length types
of all solutions of rank $n-1$ of a pair of equations
are covered by a finite union of $(n-1)$-dimensional subspaces
if the equations are not equivalent on solutions of rank $n-1$.
This means that the solution sets of pairs of equations
are in some sense more structured
than the solution sets of single equations.
This theorem is the key to proving the remaining results.

We begin Section \ref{sect:indsyst} by proving a theorem about
unbalanced equations. This gives a considerably simpler reproof and
a generalization of a result in \cite{HaNo03}.
Finally, we return to the question about sizes of independent systems.
There is a trivial bound for the size of a system
depending on the length of the longest equation,
because there are only exponentially many equations of a fixed length.
We prove that if the system is independent
even when considering only solutions of rank $n-1$,
then there is an upper bound for the size of the system
depending quadratically on the length of the shortest equation.
Even though it does not give a fixed bound even in the case of three unknowns,
it is a first result of its type --
hence opening, we hope, a new avenue for future research.

\section{Basic Theorems} \label{sect:basic}

Let $|w|$ be the length of a word $w$ and
$|w|_a$ be the number of occurrences of a letter $a$ in $w$.
The set of nonnegative integers is denoted by $\No$ and
the set of positive integers by $\Ni$.
The empty word is denoted by $\eps$.

In this section we give proofs for some well-known results.
These serve as examples of the polynomial methods used.
Even though the standard proofs of these are simple,
we hope that the proofs given here illustrate
how properties of words can be formulated and proved in terms of polynomials.

Let $\Sigma \subset \Ni$ be an alphabet of numbers.
For a word
\begin{math}
    w = a_{0} \dots a_{n-1} \in \Sigma^n
\end{math}
we define a polynomial
\begin{equation*}
    \px{w} = a_{0} + a_1 X^{1} + \dots + a_{n-1} X^{n-1}
\end{equation*}
and, if $n = |w| > 0$, a rational function
\begin{equation*}
    \rx{w} = \frac{\px{w}}{X^n - 1} .
\end{equation*}
The mapping $w \mapsto \px{w}$ is an injection from words to polynomials.
Here the assumption $0 \notin \Sigma$ is needed;
if injectivity of $\px{w}$ would not be needed,
then also 0 could be a letter.
If $w_1, \dots, w_m \in \Sigma^*$, then
\begin{equation} \label{eq:prodp}
    \px{w_1 \dots w_m}
    = \px{w_1} + \px{w_2} X^{|w_1|}
        + \dots + \px{w_m} X^{|w_1 \dots w_{m-1}|} ,
\end{equation}
and if $w_1, \dots, w_m \in \Sigma^+$, then
\begin{align*}
    \px{w_1 \dots w_m}
    =& \rx{w_1} (X^{|w_1|} - 1)
        + \rx{w_2} (X^{|w_1 w_2|} - X^{|w_1|}) \\
        &+ \dots + \rx{w_m}
            (X^{|w_1 \dots w_{m}|} - X^{|w_1 \dots w_{m-1}|}) .
\end{align*}
If $w \in \Sigma^+$ and $k \in \No$, then
\begin{equation*}
    \px{w^k} = \px{w} \frac{X^{k|w|} - 1}{X^{|w|} - 1}
    = \rx{w} (X^{k|w|} - 1) .
\end{equation*}

The polynomial $\px{w}$ can be viewed as
a characteristic polynomial of the word $w$.
A polynomial or formal power series obtained from a sequence in this way
is sometimes known as the generating function or $z$-transform of the sequence.
We could also replace $X$ with a suitable number $b$ and
get a number whose reverse $b$-ary representation is $w$.
Or we could let the coefficients of $\px{w}$
be from some other commutative ring than $\Z$.
Similar ideas have been used to analyze words in many places,
see, e.g., \cite{Ku97} and \cite{Sa85}.

\begin{example}
If $w = 1212$, then
\begin{math}
    \px{w} = 1 + 2X + X^2 + 2X^3
\end{math}
and
\begin{equation*}
    \rx{w} = \frac{1 + 2 X + X^2 + 2 X^3}{X^4 - 1}
    = \frac{(1 + X^2)(1 + 2 X)}{(X^2 + 1)(X^2 - 1)}
    = \frac{1 + 2 X^2}{X^2 - 1}.
\end{equation*}
\end{example}

A word $w \in \Sigma^+$ is \emph{primitive}
if it is not of the form $u^k$ for any $k > 1$.
If $w = u^k$ and $u$ is primitive,
then $u$ is a \emph{primitive root} of $w$.

\begin{lemma} \label{lem:primdiv}
If $w$ is primitive,
then $\px{w}$ is not divisible by any polynomial of the form
\begin{math}
    (X^{|w|} - 1) / (X^{n} - 1),
\end{math}
where $n < |w|$ is a divisor of $|w|$.
\end{lemma}

\begin{proof}
If $\px{w}$ is divisible by $(X^{|w|} - 1) / (X^n - 1)$,
then there are numbers $a_0, \dots, a_{n-1}$ such that
\begin{equation*}
\begin{split}
    \px{w}
    &= (a_{0} + a_1 X^{1} + \dots + a_{n-1} X^{n-1})
        \frac{X^{|w|} - 1}{X^n - 1} \\
    &= (a_{0} + a_1 X^{1} + \dots + a_{n-1} X^{n-1})
        (1 + X^n + \dots + X^{|w|-n}) ,
\end{split}
\end{equation*}
so $w = (a_{0} \dots a_{n-1})^{|w|/n}$.
\end{proof}

The next two theorems are among
the most basic and well-known results in combinatorics on words
(except for item \eqref{item:r} of Theorem \ref{thm:commutation},
which, however, appeared in \cite{HoKo09} in a slightly different form).

\begin{theorem}
Every nonempty word has a unique primitive root.
\end{theorem}

\begin{proof}
Let $u^m = v^n$, where $u$ and $v$ are primitive.
We need to show that $u = v$.
We have
\begin{equation*}
    \px{u} \frac{X^{m|u|} - 1}{X^{|u|} - 1} = \px{u^m}
    = \px{v^n} = \px{v} \frac{X^{n|v|} - 1}{X^{|v|} - 1} .
\end{equation*}
Because $m|u| = n|v|$, we get
\begin{math}
    \px{u} (X^{|v|} - 1) = \px{v} (X^{|u|} - 1) .
\end{math}
If $d = \gcd(|u|, |v|)$, then $\gcd(X^{|u|} - 1, X^{|v|} - 1) = X^d - 1$.
Thus $\px{u}$ must be divisible by $(X^{|u|} - 1) / (X^d - 1)$ and
$\px{v}$ must be divisible by $(X^{|v|} - 1) / (X^d - 1)$.
By Lemma \ref{lem:primdiv},
both $u$ and $v$ can be primitive only if $|u| = d = |v|$.
\end{proof}

The primitive root of a word $w \in \Sigma^+$ is denoted by $\rho(w)$.

\begin{theorem} \label{thm:commutation}
For $u, v \in \Sigma^+$, the following are equivalent:
\begin{enumerate}
\item $\rho(u) = \rho(v)$, \label{item:rho}
\item if $U, V \in \{u, v\}^*$ and $|U| = |V|$, then $U = V$, \label{item:all}
\item $u$ and $v$ satisfy a nontrivial relation, \label{item:exists}
\item $\rx{u} = \rx{v}$. \label{item:r}
\end{enumerate}
\end{theorem}

\begin{proof}
(\ref{item:rho}) $\Rightarrow$ (\ref{item:all}):
\begin{math}
    U = \rho(u)^{|U| / |\rho(u)|}
    = \rho(u)^{|V| / |\rho(u)|} = V .
\end{math}

(\ref{item:all}) $\Rightarrow$ (\ref{item:exists}):
Clear.

(\ref{item:exists}) $\Rightarrow$ (\ref{item:r}):
Let
\begin{math}
    u_1 \dots u_m = v_1 \dots v_n,
\end{math}
where $u_i, v_j \in \{u,v\}$.
Then
\begin{equation*}
\begin{split}
    0 =& \px{u_1 \dots u_m} - \px{v_1 \dots v_n} \\
    =& \rx{u_1} (X^{|u_1|} - 1)
        + \dots + \rx{u_m}
            (X^{|u_1 \dots u_{m}|} - X^{|u_1 \dots u_{m-1}|}) \\
        &- \rx{v_1} (X^{|v_1|} - 1)
        - \dots - \rx{v_n}
            (X^{|v_1 \dots v_{n}|} - X^{|v_1 \dots v_{n-1}|}) \\
    =& \rx{u} p - \rx{v} p
\end{split}
\end{equation*}
for some polynomial $p$.
If $m \ne n$ or $u_i \ne v_i$ for some $i$,
then $p \ne 0$, and thus $\rx{u} = \rx{v}$.

(\ref{item:r}) $\Rightarrow$ (\ref{item:rho}):
We have
\begin{math}
    \px{u^{|v|}}
    = \rx{u} (X^{|u||v|} - 1)
    = \rx{v} (X^{|u||v|} - 1)
    = \px{v^{|u|}} ,
\end{math}
so $u^{|v|} = v^{|u|}$ and
\begin{math}
    \rho(u) = \rho(u^{|v|}) = \rho(v^{|u|}) = \rho(v) .
\end{math}
\end{proof}

Similarly, polynomials can be used to give a simple proof
for the theorem of Fine and Wilf.
In fact, one of the original proofs in \cite{FiWi65} uses power series.
The proof given here is essentially this original proof
formulated in terms of our polynomials.
Algebraic techniques have also been used
to prove variations of this theorem \cite{MiShWa01}.

\begin{theorem}[Fine and Wilf] \label{thm:finewilf}
If $u^i$ and $v^j$ have a common prefix of length
\begin{math}
    |u| + |v| - \gcd(|u|, |v|),
\end{math}
then $\rho(u) = \rho(v)$.
\end{theorem}

\begin{proof}
Let $\gcd(|u|, |v|) = d$, $\lcm(|u|, |v|) = m$,
$m / |u| = r$ and $m / |v| = s$.
If $\rho(u) \ne \rho(v)$, then $u^r \ne v^s$,
so $u^r$ and $v^s$ have a maximal common prefix of length $k < m$.
This means that
\begin{equation*}
\begin{split}
    \px{u^r} - \px{v^s}
    &= \frac{X^{r|u|} - 1}{X^{|u|} - 1} \px{u}
        - \frac{X^{s|v|} - 1}{X^{|v|} - 1} \px{v} \\
    &= \frac{(X^{m} - 1)(X^d - 1)}{(X^{|u|} - 1)(X^{|v|} - 1)}
        \left( \frac{X^{|v|} - 1}{X^d - 1} \px{u}
            - \frac{X^{|u|} - 1}{X^d - 1} \px{v} \right)
\end{split}
\end{equation*}
is divisible by $X^k$, but not by $X^{k+1}$, so also the polynomial
\begin{equation*} \label{eq:rdiff1}
    \frac{X^{|v|} - 1}{X^d - 1} \px{u}
    - \frac{X^{|u|} - 1}{X^d - 1} \px{v}
\end{equation*}
is divisible by $X^k$, but not by $X^{k+1}$.
Thus $k$ can be at most the degree of this polynomial,
which is at most $|u| + |v| - d - 1$.
\end{proof}

\section{Solutions of Fixed Length} \label{sect:fixedlength}

In this section we apply polynomial techniques to word equations.
From now on, we will assume that there are $n$ unknowns,
they are ordered as $x_1, \dots, x_n$
and $\Xi$ is the set of these unknowns.

A (coefficient-free) \emph{word equation} $u = v$ on $n$ unknowns
consists of two words $u, v \in \Xi^*$.
A \emph{solution} of this equation
is any morphism $h: \Xi^* \to \Sigma^*$ such that $h(u) = h(v)$.
The equation is \emph{trivial} if $u$ and $v$ are the same word.

The (combinatorial) \emph{rank} of a morphism $h$ is the smallest
number $r$ for which there is a set $A$ of $r$ words such that $h(x)
\in A^*$ for every unknown $x$. A morphism of rank at most one is
\emph{periodic}.

Let $h: \Xi^* \to \Sigma^*$ be a morphism.
The \emph{length type} of $h$ is the vector
\begin{equation*}
    L = (|h(x_1)|, \dots, |h(x_n)|) \in \No^n.
\end{equation*}
This length type $L$ determines a morphism
\begin{equation*}
    \len{L}: \Xi^* \to \No, \len{L}(w) = |h(w)|.
\end{equation*}
It is important that $\len{L}$ depends only on $L$ and not on $h$.

If $E$ is a word equation,
the set of its solutions is denoted by $\sol{}{}(E)$,
the set of solutions of rank $r$ by $\sol{r}{}(E)$,
the set of solutions of length type $L$ by $\sol{}{L}(E)$ and
the set of solutions of rank $r$ and length type $L$ by $\sol{r}{L}(E)$.
These can be naturally generalized for systems of equations.
For example, if $E_1$ and $E_2$ are word equations,
then $\sol{}{}(E_1, E_2) = \sol{}{}(E_1) \cap \sol{}{}(E_2)$.

For a word equation
\begin{math}
    E: y_1 \dots y_k = z_1 \dots z_l
\end{math}
(where $y_i, z_i \in \Xi$), a variable $x \in \Xi$ and a length type $L$, let
\begin{equation*}
    \qx{E}{x}{L} = \sum_{y_i = x} X^{\len{L}(y_1 \dots y_{i-1})}
        - \sum_{z_i = x} X^{\len{L}(z_1 \dots z_{i-1})} .
\end{equation*}
Informally, this polynomial encodes the positions of $x$ in the equation $E$.

\begin{theorem} \label{thm:weqpeq}
A morphism $h: \Xi^* \to \Sigma^*$ of length type $L$
is a solution of an equation $E: u = v$ if and only if
\begin{equation*}
    \sum_{x \in \Xi} \qx{E}{x}{L} \px{h(x)} = 0.
\end{equation*}
\end{theorem}

\begin{proof}
Now $h(u) = h(v)$ if and only if $\px{h(u)} = \px{h(v)}$,
and the polynomial $\px{h(u)} - \px{h(v)}$
can be written as $\sum_{x \in \Xi} \qx{E}{x}{L} \px{h(x)}$
by \eqref{eq:prodp}.
\end{proof}

Theorem \ref{thm:weqpeq} means that if we fix a length type $L$,
then we can turn a word equation into a linear equation
where the polynomials $\qx{E}{x}{L}$ are the coefficients.
A solution for this linear equation is an $n$-dimensional vector
over the field of rational functions,
and $h \in \sol{}{L}(E)$ corresponds to a solution
\begin{math}
    (\px{h(x_1)}, \dots \px{h(x_n)})
\end{math}
of the linear equation.

\begin{example}
Let $\Xi = \{x, y, z\}$, $E: xyz = zxy$ and $L = (1, 1, 2)$.
Then
\begin{equation*}
    \qx{E}{x}{L} = 1 - X^2, \qquad
    \qx{E}{y}{L} = X - X^3, \qquad
    \qx{E}{z}{L} = X^2 - 1.
\end{equation*}
If $h$ is the morphism defined by $h(x) = 1$, $h(y) = 2$ and $h(z) =
12$, then $h$ is a solution of $E$ and
\begin{equation*}
\begin{split}
    & \qx{E}{x}{L} \px{h(x)} + \qx{E}{y}{L} \px{h(y)}
        + \qx{E}{z}{L} \px{h(z)} \\
    =& (1 - X^2) \cdot 1 + (X - X^3) \cdot 2 + (X^2 - 1) (1 + 2X)
    = 0.
\end{split}
\end{equation*}
\end{example}

At this point we start using linear algebra.
We will do this over two fields:
The field of rational numbers
(for the first time in Lemma \ref{lem:rdim})
and the field of rational functions
(for the first time in Lemma \ref{lem:rdim2}).
We start with an example.

\begin{example}
Consider the morphism
$h: \{x_1, x_2, x_3\}^* \to \{1, 2\}^*$
of rank 2 defined by
\begin{math}
    h(x_1) = 1, h(x_2) = 2, h(x_3) = 12 .
\end{math}
If $h$ is a solution of an equation $E$,
then so is $g \circ h$ for every morphism $g: \{1, 2\}^* \to \{1, 2\}^*$.
The length type of $g \circ h$ is
\begin{equation*}
    (|g(1)|, |g(2)|, |g(12)|)
    = |g(1)| \cdot (1, 0, 1) + |g(2)| \cdot (0, 1, 1) .
\end{equation*}
Because the vectors $(1, 0, 1)$ and $(0, 1, 1)$ are linearly independent,
these length types essentially form a two-dimensional space
(of course $|g(1)|$ and $|g(2)|$ are nonnegative integers,
so the length types don't form the whole space).
This observation is formalized and generalized in Lemma \ref{lem:rdim}.
\end{example}

A morphism $\phi: \Xi^* \to \Xi^*$ is an \emph{elementary transformation}
if there are two unknowns $x, y \in \Xi$
so that $\phi(y) \in \{xy, x\}$ and
$\phi(z) = z$ for $z \in \Xi \smallsetminus \{y\}$.
If $\phi(y) = xy$, then $\phi$ is \emph{regular},
and if $\phi(y) = x$, then $\phi$ is \emph{singular}.
The next lemma follows immediately from results in \cite{Lo83}.

\begin{lemma} \label{lem:elemtrans}
Every solution $h$ of an equation $E$ has a factorization
\begin{math}
    h = \theta \circ \phi \circ \alpha,
\end{math}
where $\alpha(x) \in \{x, \eps\}$ for all $x \in \Xi$,
\begin{math}
    \phi = \phi_m \circ \dots \circ \phi_1,
\end{math}
every $\phi_i$ is an elementary transformation,
$\phi \circ \alpha$ is a solution of $E$ and
$\theta(x) \ne \eps$ for all $x \in \Xi$.
If $\alpha(x) = \eps$ for $s$ unknowns $x$ and
$t$ of the $\phi_i$ are singular,
then the rank of $\phi \circ \alpha$ is $n - s - t$.
\end{lemma}

\begin{lemma} \label{lem:rdim}
Let $E$ be an equation on $n$ unknowns and
let $h \in \sol{r}{L}(E)$.
There is an $r$-dimensional subspace $V$ of $\Q^n$ containing $L$ such that
the set of those length types of morphisms in $\sol{r}{}(E)$ that are in $V$
is not covered by any finite union of $(r-1)$-dimensional spaces.
\end{lemma}

\begin{proof}
For arbitrary morphisms
$F: \Xi^* \to \Xi^*$ and $G: \Xi^* \to \Sigma^*$,
let
\begin{math}
    L_{G} = (|G(x_1)|, \dots, |G(x_n)|)^T
\end{math}
be the length type of $G$ as a column vector and let
\begin{math}
    A_{F} = ( |F(x_i)|_{x_j} )
\end{math}
be an $n \times n$ matrix.
Then
\begin{math}
    L_{G \circ F} = A_{F} L_{G} .
\end{math}
More generally, if $F_1, \dots, F_m$ are morphisms $\Xi^* \to \Xi^*$,
then
\begin{equation*}
    L_{G \circ F_m \circ \dots \circ F_1}
    = A_{F_1} \dots A_{F_m} L_{G} .
\end{equation*}

Let
\begin{math}
    h = \theta \circ \phi_m \circ \dots \circ \phi_1 \circ \alpha
\end{math}
as in Lemma \ref{lem:elemtrans}.
Let $f = \phi_m \circ \dots \circ \phi_1 \circ \alpha$.
The rank of $f$ is $n - s - t \geq r$
if $s$ and $t$ are as in Lemma \ref{lem:elemtrans}.
The morphism $g \circ f$ is a solution of $E$
for every morphism $g: \Xi^* \to \Sigma^*$.
The length type of $g \circ f$ is
\begin{math}
    L_{g \circ f}
    = L_{g \circ \phi_m \circ \dots \circ \phi_1 \circ \alpha}
    = A L_g ,
\end{math}
where
\begin{math}
    A = A_\alpha A_{\phi_1} \dots A_{\phi_m} .
\end{math}
To prove the theorem, it needs to be shown that
the rank of $A$ is at least $r$.
This can be done by determining the ranks of the matrices
$A_{\alpha}$ and $A_{\phi_k}$.

The matrix $A_{\alpha}$ is a diagonal matrix
and the $i$th element on the diagonal is 0
if $\alpha(x_i) = \eps$ and 1 otherwise.
Thus the rank of $A_{\alpha}$ is $n - s$.

If $\phi$ is the elementary transformation defined by $\phi(x_1) = x_2 x_1$,
then
\begin{equation*}
    A_{\phi} =
    \begin{pmatrix}
        1 & 1 & 0 & \dots & 0 \\
        0 & 1 & 0 & \dots & 0 \\
        0 & 0 & 1 & \dots & 0 \\
        \dots \\
        0 & 0 & 0 & \dots & 1
    \end{pmatrix}
\end{equation*}
is a matrix of rank $n$
(this is an identity matrix
except for the second element on the first row).
In general, the rank of $A_{\phi}$ is $n$
for every regular elementary transformation $\phi$.

If $\phi$ is the elementary transformation defined by $\phi(x_1) = x_2$,
then
\begin{equation*}
    A_{\phi} =
    \begin{pmatrix}
        0 & 1 & 0 & \dots & 0 \\
        0 & 1 & 0 & \dots & 0 \\
        0 & 0 & 1 & \dots & 0 \\
        \dots \\
        0 & 0 & 0 & \dots & 1
    \end{pmatrix}
\end{equation*}
is a matrix of rank $n - 1$
(this is an identity matrix
except for the first two elements on the first row).
In general, the rank of $A_{\phi}$ is $n - 1$
for every singular elementary transformation $\phi$.

The rank of $A_{\alpha}$ is $n - s$,
$t$ of the matrices $A_{\phi_k}$ have rank $n - 1$ and
the rest have rank $n$.
Thus the rank of $A$ is at least $n - s - t$, which is at least $r$.
\end{proof}

\begin{lemma} \label{lem:rdim2}
Let $E$ be an equation on $n$ unknowns and let $h \in \sol{r}{L}(E)$.
There are morphisms $f: \Xi^* \to \Xi^*$ and $\theta: \Xi^* \to \Sigma^*$ and
polynomials $p_{ij}$ such that the following conditions hold:
\begin{enumerate}
\item $h = \theta \circ f$,
\item $f$ is a solution of $E$,
\item $\theta(x) \ne \eps$ for every $x \in \Xi$,
\item $\px{(g \circ f)(x_i)} = \sum p_{ij} \px{g(x_j)}$ for all
    $i, j$ if $g: \Xi^* \to \Sigma^*$ is a morphism of the same
    length type as $\theta$,
\item $r$ of the vectors $(p_{1j}, \dots, p_{nj}) \in \Q(X)^n$, where
    $j = 1, \dots, n$, are linearly independent.
\end{enumerate}
\end{lemma}

\begin{proof}
The proof is quite similar to the proof of Lemma \ref{lem:rdim}.

For arbitrary morphisms
$F: \Xi^* \to \Xi^*$ and $G: \Xi^* \to \Sigma^*$
and length type $L$, define an $n$-dimensional column vector
\begin{math}
    \px{G} = (\px{G(x_1)}, \dots, \px{G(x_n)})^T
\end{math}
and an $n \times n$ polynomial matrix
\begin{math}
    B_{F, L} = ( b_{ij} ) ,
\end{math}
where
\begin{equation*}
    b_{ij} = \sum_{u x_j \leq F(x_i)} X^{\len{L}(u)} .
\end{equation*}
If $L$ is the length type of $G$, then
\begin{math}
    \px{G \circ F} = B_{F, L} \px{G} .
\end{math}
More generally, if $F_1, \dots, F_m$ are morphisms $\Xi^* \to \Xi^*$
and $L_k$ is the length type of $G \circ F_m \circ \dots \circ F_{k+1}$,
then
\begin{equation*}
    \px{G \circ F_m \circ \dots \circ F_1}
    = B_{F_1, L_1} \dots B_{F_m, L_m} \px{G} .
\end{equation*}
The matrices $B_{F, L}$ will be used to define the polynomials $p_{ij}$.

Let
\begin{math}
    h = \theta \circ \phi_m \circ \dots \circ \phi_1 \circ \alpha
\end{math}
as in Lemma \ref{lem:elemtrans}.
Let $f = \phi_m \circ \dots \circ \phi_1 \circ \alpha$.
The first three conditions are satisfied by $\theta$ and $f$.
The rank of $f$ is $n - s - t \geq r$
if $s$ and $t$ are as in Lemma \ref{lem:elemtrans}.

Let $L$ be the length type of $\theta$
and let $g$ be a morphism of length type $L$.
Then
\begin{math}
    \px{g \circ f}
    = \px{g \circ \phi_m \circ \dots \circ \phi_1 \circ \alpha}
    = B \px{g} ,
\end{math}
where
\begin{math}
    B = B_{\alpha, L_0} B_{\phi_1, L_1} \dots B_{\phi_m, L_m}
\end{math}
and $L_k$ is the length type of $g \circ \phi_m \circ \dots \circ \phi_{k+1}$.
Let
\begin{math}
    B = (p_{ij}) .
\end{math}
Then the fourth condition holds, because $\px{g \circ f} = B \px{g}$.
%($p_{ij}$ ``encodes'' the positions of the word $g(x_j)$ in $h(x_i)$).

To prove that the last condition holds,
it must be proved that the rank of the matrix $B$ is at least $r$.
This can be done by determining the ranks of the matrices
$B_{\alpha, L}$ and $B_{\phi_k, L}$.

The matrix $B_{\alpha, L}$ is a diagonal matrix
and the $i$th element on the diagonal is 0
if $\alpha(x_i) = \eps$ and 1 otherwise.
Thus the rank of $B_{\alpha, L}$ is $n - s$.

If $\phi$ is the elementary transformation defined by $\phi(x_1) = x_2 x_1$,
then
\begin{equation*}
    B_{\phi, L} =
    \begin{pmatrix}
        X^{\len{L}(x_2)} & 1 & 0 & \dots & 0 \\
        0 & 1 & 0 & \dots & 0 \\
        0 & 0 & 1 & \dots & 0 \\
        \dots \\
        0 & 0 & 0 & \dots & 1
    \end{pmatrix}
\end{equation*}
is a matrix of rank $n$
(this is an identity matrix
except for the first two elements on the first row).
In general, the rank of $B_{\phi, L}$ is $n$
for every regular elementary transformation $\phi$.

If $\phi$ is the elementary transformation defined by $\phi(x_1) = x_2$,
then
\begin{equation*}
    B_{\phi, L} =
    \begin{pmatrix}
        0 & 1 & 0 & \dots & 0 \\
        0 & 1 & 0 & \dots & 0 \\
        0 & 0 & 1 & \dots & 0 \\
        \dots \\
        0 & 0 & 0 & \dots & 1
    \end{pmatrix}
\end{equation*}
is a matrix of rank $n - 1$
(again, this is an identity matrix
except for the first two elements on the first row).
In general, the rank of $B_{\phi, L}$ is $n - 1$
for every singular elementary transformation $\phi$.

The rank of $B_{\alpha, L_0}$ is $n - s$,
$t$ of the matrices $B_{\phi_k, L_k}$ have rank $n - 1$ and
the rest have rank $n$.
Thus the rank of $B$ is at least $n - s - t$, which is at least $r$.
\end{proof}

With the help of these lemmas,
we are going to analyze solutions of some fixed length type.
Principal (or fundamental) solutions,
which were implicitly present in the previous lemmas (see \cite{Lo83}),
have been used in connection with fixed lengths
also in \cite{Ho00} and \cite{Ho01}.

\begin{theorem} \label{thm:rank}
Let $E_1, \dots, E_m$ be a system of equations on $n$ unknowns and
let $L \in \No^n$.
Let
\begin{math}
    q_{ij} = \qx{E_i}{x_j}{L}.
\end{math}
If $\sol{r}{L}(E_1, \dots, E_m) \ne \varnothing$,
then the rank of the $m \times n$ polynomial matrix $(q_{ij})$
is at most $n - r$.
If the rank of the matrix is 1,
at most one component of $L$ is zero and
the equations are nontrivial,
then
\begin{math}
    \sol{}{L}(E_1) = \dots = \sol{}{L}(E_m) .
\end{math}
\end{theorem}

\begin{proof}
Let $h \in \sol{r}{L}(E_1, \dots, E_m)$.
If $r = 1$, the first claim follows from Theorem \ref{thm:weqpeq},
so assume that $r > 1$.
Let $E$ be an equation that has the same nonperiodic solutions as the system.
Lemma \ref{lem:rdim2} will be used for this equation.
Fix $k$ and let $g_1: \Xi^* \to \Sigma^*$ be the morphism determined by
\begin{math}
    g_1(x_i) = 1^{|\theta(x_i)|}
\end{math}
for all $i$
and let $g_2: \Xi^* \to \Sigma^*$ be the morphism determined by
\begin{math}
    g_2(x_k) = 21^{|\theta(x_k)| - 1}
\end{math}
and
\begin{math}
    g_2(x_i) = 1^{|\theta(x_i)|}
\end{math}
for all $i \ne k$.
Then $g_1 \circ f$ and $g_2 \circ f$ are solutions of every $E_l$, so
\begin{equation*}
    \sum_{i=1}^n \qx{E_l}{x_i}{L} \px{(g_1 \circ f)(x_i)} = 0
    \qquad \text{and} \qquad
    \sum_{i=1}^n \qx{E_l}{x_i}{L} \px{(g_2 \circ f)(x_i)} = 0
\end{equation*}
for all $l$ by Theorem \ref{thm:weqpeq}.
Because also
\begin{math}
    \px{(g_1 \circ f)(x_i)} = \sum_{j=1}^n p_{ij} \px{g_1(x_j)}
\end{math}
and
\begin{math}
    \px{(g_2 \circ f)(x_i)} = \sum_{j=1}^n p_{ij} \px{g_2(x_j)} ,
\end{math}
we get
\begin{equation*}
\begin{split}
    0 =& \sum_{i=1}^n \qx{E_l}{x_i}{L}
        (\px{(g_2 \circ f)(x_i)} - \px{(g_1 \circ f)(x_i)}) \\
    =& \sum_{i=1}^n \qx{E_l}{x_i}{L}
        \sum_{j=1}^n p_{ij} (\px{g_2(x_j)} - \px{g_1(x_j)})
    = \sum_{i=1}^n \qx{E_l}{x_i}{L} p_{ik}
\end{split}
\end{equation*}
for all $l$.
Thus the vectors $(p_{1j}, \dots, p_{nj})$ are solutions
of the linear system of equations determined by the matrix $(q_{ij})$.
Because at least $r$ of these vectors are linearly independent,
the rank of the matrix is at most $n - r$.

If at most one component of $L$ is zero and the equations are nontrivial,
then all rows of the matrix are nonzero.
If also the rank of the matrix is 1,
then all rows are multiples of each other and
the second claim follows by Theorem \ref{thm:weqpeq}.
\end{proof}

\section{Applications} \label{sect:appl}

Based on Theorem \ref{thm:rank},
the polynomial and linear algebraic methods will be developed further
in Section \ref{sect:solsets}.
However, Theorem \ref{thm:rank} is already strong enough
to provide reproofs, generalizations and improvements of some results.

The \emph{graph} of a system of word equations is the graph
where $\Xi$ is the set of vertices
and there is an edge between $x$ and $y$
if one of the equations in the system is of the form $x \dd = y \dd$.
The following well-known theorem can be proved
with the help of Theorem \ref{thm:rank}.

\begin{theorem}[Graph lemma] \label{thm:graph}
Consider a system of equations whose graph has $r$ connected components.
If $h$ is a solution of this system and $h(x_i) \ne \eps$ for all $i$,
then the rank of $h$ is at most $r$.
\end{theorem}

\begin{proof}
We can assume that the connected components are
\begin{equation*}
    \{x_1, \dots, x_{i_2-1}\},
    \{x_{i_2}, \dots, x_{i_3-1}\},
    \dots,
    \{x_{i_r}, \dots, x_n\}
\end{equation*}
and the equations are
\begin{equation*}
    x_j \dd = x_{k_j} \dd,
\end{equation*}
where
\begin{math}
    j \in \{1, \dots, n\} \smallsetminus \{1, i_2, \dots, i_r\}
\end{math}
and $k_j < j$.
Let $q_{ij}$ be as in Theorem \ref{thm:rank}.
If we remove the columns $1, i_2, \dots, i_r$
from the $(n-r) \times n$ matrix $(q_{ij})$,
we obtain a square matrix $M$
where the diagonal elements are not divisible by $X$,
but all elements above the diagonal are divisible by $X$.
This means that $\det(M)$ is not divisible by $X$, so $\det(M) \ne 0$.
Thus the rank of the matrix $(q_{ij})$ is $n-r$ and
$h$ has rank at most $r$ by Theorem \ref{thm:rank}.
\end{proof}

The next theorem generalizes a result from \cite{CzKa07}
for more than three unknowns.

\begin{theorem}
If a pair of nontrivial equations on $n$ unknowns
has a solution $h$ of rank $n-1$
where no two of the unknowns commute,
then there is a number $k \geq 1$ such that the equations are of the form
\begin{math}
    x_1 \dd = x_2^k x_3 \dd .
\end{math}
\end{theorem}

\begin{proof}
By Theorem \ref{thm:graph},
the equations must be of the form $x_1 \dd = x_2 \dd$.
Let them be
\begin{equation*}
    x_1 u y \dd = x_2 v z \dd
    \qquad \text{and} \qquad
    x_1 u' y' \dd = x_2 v' z' \dd,
\end{equation*}
where $u, v, u', v' \in \{x_1, x_2\}^*$
and $y, z, y', z' \in \{x_3, \dots, x_n\}$.
It can be assumed that $z = x_3$ and
\begin{equation*}
    |h(x_2 v)| \leq |h(x_1 u)|, |h(x_1 u')|, |h(x_2 v')|.
\end{equation*}
If it were $|h(x_1 u)| = |h(x_2 v)|$,
then $h(x_1)$ and $h(x_2)$ would commute,
so $|h(x_1 u)| > |h(x_2 v)|$.
If $v$ would contain $x_1$,
then $h(x_1)$ and $h(x_2)$ would commute by Theorem \ref{thm:finewilf},
so $v = x_2^{k-1}$ for some $k \geq 1$.

Let $L$ be the length type of $h$ and
let $q_{ij}$ be as in Theorem \ref{thm:rank}.
By Theorem \ref{thm:rank},
the rank of the matrix $(q_{ij})$ must be 1 and thus
\begin{math}
    q_{12} q_{23} - q_{13} q_{22} = 0.
\end{math}
The term of
\begin{math}
    q_{13} q_{22}
\end{math}
of the lowest degree is $X^{|h(x_2^k)|}$.
The same must hold for
\begin{math}
    q_{12} q_{23},
\end{math}
and thus the term of $q_{23}$ of the lowest degree must be $-X^{|h(x_2^k)|}$.
We know that $x_2 v = x_2^k$ and assumed that $|h(x_2 v)| \leq |h(x_2 v')|$.
If it were $|h(x_2 v)| < |h(x_2 v')|$,
then $h(x_3)$ would start in $h(x_2 v' z' \dots)$
before the end of $h(x_2 v')$,
which is not possible.
This means that $|h(x_2 v')| = |h(x_2^k)| \leq |h(x_1 u')|$ and $z' = x_3$.
As above, we conclude that $|h(x_2 v')| < |h(x_1 u')|$,
$v'$ cannot contain $x_1$ and $v' = x_2^{k-1}$.
\end{proof}

It was proved in \cite{Ko98} that if
\begin{equation*}
      s_0 u_1^i s_1 \dots u_m^i s_m
    = t_0 v_1^i t_1 \dots v_n^i t_n
\end{equation*}
holds for $m+n+3$ consecutive values of $i$,
then it holds for all $i$.
By using similar ideas as in Theorem \ref{thm:rank},
we improve this bound to $m+n$ and
prove that the values do not need to be consecutive.
In \cite{Ko98} it was also stated that
the arithmetization and matrix techniques in \cite{Tu87}
would give a simpler proof of a weaker result.
Similar questions have been studied in \cite{HoKo07} and
there are relations to independent systems \cite{Pl03}.

\begin{theorem}
Let $m,n \geq 1$, $s_j, t_j \in \Sigma^*$ and $u_j, v_j \in \Sigma^+$.
Let
\begin{equation*}
    U_i = s_0 u_1^i s_1 \dots u_m^i s_m
    \qquad \text{and} \qquad
    V_i = t_0 v_1^i t_1 \dots v_n^i t_n .
\end{equation*}
If $U_i = V_i$ holds for $m+n$ values of $i$, then it holds for all $i$.
\end{theorem}

\begin{proof}
The equation $U_i = V_i$ is equivalent to $\px{U_i} - \px{V_i} = 0$.
Because
\begin{equation*}
\begin{split}
    \px{U_i} =&
    \sum_{j=1}^m \left( \px{s_{j-1}} + \px{u_j}
        \frac{X^{i|u_j|} - 1}{X^{|u_j|} - 1} X^{|s_{j-1}|} \right)
        X^{i |u_{1} \dots u_{j-1}| + |s_{0} \dots s_{j-2}|} \\
    &+ \px{s_m} X^{i |u_{1} \dots u_{m}| + |s_{0} \dots s_{m-1}|}
\end{split}
\end{equation*}
and $\px{V_i}$ is of a similar form,
this equation can be written as
\begin{equation} \label{eq:1}
    \sum_{j=0}^{m} y_j X^{i |u_1 \dots u_j|}
    + \sum_{k \in K} z_k X^{i |v_1 \dots v_k|} = 0,
\end{equation}
where $y_j, z_k$ are some polynomials that do not depend on $i$
and $K$ is the set of those $k \in \{0, \dots n\}$
for which $|v_1 \dots v_k|$ is not any of the numbers $|u_1 \dots u_j|$
($j = 0, \dots, m$).
If $U_{i_1} = V_{i_1}$ and $U_{i_2} = V_{i_2}$, then
\begin{equation*}
    (i_1 - i_2) |u_1 \dots u_m| = |U_{i_1}| - |U_{i_2}|
    = |V_{i_1}| - |V_{i_2}| = (i_1 - i_2) |v_1 \dots v_n| .
\end{equation*}
Thus $|u_1 \dots u_m| = |v_1 \dots v_n|$ and the size of $K$ is at most $n-1$.
If \eqref{eq:1} holds for $m + 1 + \# K \leq m+n$ values of $i$,
it can be viewed as a system of equations where $y_j, z_k$ are unknowns.
The coefficients of this system form a generalized Vandermonde matrix
whose determinant is nonzero,
so the system has a unique solution $y_j = z_k = 0$ for all $j, k$.
This means that \eqref{eq:1} holds for all $i$ and $U_i = V_i$ for all $i$.
\end{proof}

\section{Sets of Solutions} \label{sect:solsets}

In this section we analyze how the polynomials
\begin{math}
    \qx{E}{x}{L}
\end{math}
behave when $L$ is not fixed.
Let
\begin{equation*}
    \lhp = \set{a_1 X_1 + \dots + a_n X_n}{a_1, \dots, a_n \in \No}
    \subset \Z[X_1, \dots, X_n]
\end{equation*}
be the additive monoid of linear homogeneous polynomials
with nonnegative integer coefficients
on the variables $X_1, \dots, X_n$.
The \emph{monoid ring} of $\lhp$ over $\Z$
is the ring formed by expressions of the form
\begin{equation*}
    a_1 X^{p_1} + \dots + a_k X^{p_k},
\end{equation*}
where $a_i \in \Z$ and $p_i \in \lhp$,
and the addition and multiplication of these generalized polynomials
is defined in a natural way.
This ring is denoted by $\Z[X;\lhp]$.
If $L \in \Z^n$,
then the value of a polynomial $p \in \lhp$
at the point $(X_1, \dots, X_n) = L$
is denoted by $p(L)$,
and the polynomial we get by making this substitution in $s \in \Z[X;\lhp]$
is denoted by $s(L)$.

The ring $\Z[X;\lhp]$ is isomorphic to
the ring $\Z[Y_1, \dots, Y_n]$ of polynomials on $n$ variables.
The isomorphism is given by $X^{X_i} \mapsto Y_i$.
However, the generalized polynomials where the exponents are in $\lhp$
are suitable for our purposes.

If $a_i \leq b_i$ for $i = 1, \dots, n$,
then we use the notation
\begin{equation*}
    a_1 X_1 + \dots + a_n X_n \preceq b_1 X_1 + \dots + b_n X_n .
\end{equation*}
If $p, q \in \lhp$ and $p \preceq q$,
then $p(L) \leq q(L)$ for all $L \in \No^n$.

For an equation $E: x_{i_1} \dots x_{i_r} = x_{j_1} \dots x_{j_s}$ we define
\begin{equation*}
    \sx{E}{x} = \sum_{x_{i_k} = x} X^{X_{i_1} + \dots + X_{i_{k-1}}}
        - \sum_{x_{j_k} = x} X^{X_{j_1} + \dots + X_{j_{k-1}}}
    \in \Z[X;\lhp] .
\end{equation*}
Then $\sx{E}{x}(L) = \qx{E}{x}{L}$.
Theorem \ref{thm:weqpeq} can be formulated
in terms of these generalized polynomials.

\begin{theorem}
A morphism $h: \Xi^* \to \Sigma^*$ of length type $L$
is a solution of an equation $E$ if and only if
\begin{equation*}
    \sum_{x \in \Xi} \sx{E}{x}(L) \px{h(x)} = 0 .
\end{equation*}
\end{theorem}

\begin{example}
Let $E: x_1 x_2 x_3 = x_3 x_1 x_2$.
Then
\begin{equation*}
    \sx{E}{x_1} = 1 - X^{X_3}, \qquad
    \sx{E}{x_2} = X^{X_1} - X^{X_1 + X_3}, \qquad
    \sx{E}{x_3} = X^{X_1 + X_2} - 1.
\end{equation*}
\end{example}

The \emph{length} of an equation $E: u = v$ is $|E| = |uv|$.
The number of occurrences of an unknown $x$ in $E$ is $|E|_x = |uv|_x$.

\begin{theorem} \label{thm:cover}
Let $E_1, E_2$ be a pair of nontrivial equations on $n$ unknowns.
Let $\sol{n-1}{}(E_1) \ne \sol{n-1}{}(E_2)$.
For some unknowns $x_k, x_l$,
the set of length types of solutions of the pair of rank $n-1$
is covered by a union of
$(|E_1|_{x_k} + |E_1|_{x_l})^2$
$(n-1)$-dimensional subspaces of $\Q^n$.
If $V_1, \dots, V_m$ is a minimal such cover and $L \in V_i$ for some $i$,
then $\sol{n-1}{L}(E_1) = \sol{n-1}{L}(E_2)$.
\end{theorem}

\begin{proof}
Let
\begin{math}
    s_{ij} = \sx{E_i}{x_j}
\end{math}
for $i = 1, 2$ and $j = 1, \dots, n$.
If all $2 \times 2$ minors of the $2 \times n$ matrix $(s_{ij})$ are zero,
then for all length types $L$ of solutions of rank $n-1$
the rank of the matrix $(q_{ij})$ in Theorem \ref{thm:rank} is 1 and
$E_1$ and $E_2$ are equivalent, which is a contradiction.
Thus there are $k,l$ such that
\begin{equation*}
    t_{kl} = s_{1k} s_{2l} - s_{1l} s_{2k} \ne 0 .
\end{equation*}
The generalized polynomial $t_{kl}$ can be written as
\begin{equation*}
   t_{kl} = \sum_{i=1}^M X^{p_i} - \sum_{i=1}^N X^{q_i},
\end{equation*}
where $p_i, q_i \in \lhp$ and $p_i \ne q_j$ for all $i,j$.
If $L$ is a length type of a solution of rank $n-1$,
then $M=N$ and $L$ must be a solution of the system of equations
\begin{equation} \label{eq:ssyst}
    p_i = q_{\sigma(i)} \qquad (i=1,\dots,M)
\end{equation}
for some permutation $\sigma$.
For every $\sigma$ the equations determine
an at most $(n-1)$-dimensional space.

Let the equations be $E_1: u_1 = v_1$ and $E_2: u_2 = v_2$.
Let
\begin{align*}
    &|u_1|_{x_k} = A ,& &|v_1|_{x_k} = A' ,& &|u_2|_{x_k} = B ,& &|v_2|_{x_k} = B', \\
    &|u_1|_{x_l} = C ,& &|v_1|_{x_l} = C' ,& &|u_2|_{x_l} = D ,& &|v_2|_{x_l} = D'.
\end{align*}
Then $s_{1k}, s_{2l}, s_{1l}, s_{2k}$ can be written as
\begin{equation*}
\begin{split}
    s_{1k} = \sum_{i=1}^{A} X^{a_i} - \sum_{i=1}^{A'} X^{a'_i},
    \quad
    s_{2l} = \sum_{i=1}^{B} X^{b_i} - \sum_{i=1}^{B'} X^{b'_i},
    \\
    s_{1l} = \sum_{i=1}^{C} X^{c_i} - \sum_{i=1}^{C'} X^{c'_i},
    \quad
    s_{2k} = \sum_{i=1}^{D} X^{d_i} - \sum_{i=1}^{D'} X^{d'_i},
\end{split}
\end{equation*}
where $a_i \preceq a_{i+1}$, $a'_i \preceq a'_{i+1}$, and so on.
The polynomials $p_i$ form a subset of the polynomials
$a_i + b_j$, $a'_i + b'_j$, $c_i + d'_j$ and $c'_i + d_j$
(the reason that they form just a subset is
that we assumed $p_i \ne q_j$ for all $i,j$).
For any $i$, let $j_i$ be the smallest index $j$
such that $a_i + b_j = p_m$ for some $m$.
Then for every $i,j,m$ such that $a_i + b_j = p_m$
we have $a_i + b_{j_i} \preceq p_m$.
We can do a similar thing for the polynomials
$a'_i, b'_i$ and $c_i, d'_i$ and $c'_i, d_i$.
In this way we obtain at most
\begin{math}
    A + A' + C + C'
\end{math}
polynomials $p_i$ such that for any $L$
the value of one of these polynomials is minimal among the values $p_i(L)$.
Similarly we obtain at most
\begin{math}
    A + A' + C + C'
\end{math}
``minimal'' polynomials $q_i$.
If $L$ satisfies one of the systems \eqref{eq:ssyst},
then the smallest of the values $p_i(L)$ must be the same as
the smallest of the values $q_i(L)$.
Thus $L$ must satisfy some equation $p_i = q_j$,
where $p_i$ and $q_j$ are some of the ``minimal'' polynomials.
There are at most
\begin{equation*}
    (A + A' + C + C')^2 = (|E_1|_{x_k} + |E_1|_{x_l})^2
\end{equation*}
possible pairs of such polynomials,
and each of them determines an $(n-1)$-dimensional space.

Consider the second claim.
Because the cover is minimal,
there is a solution of rank $n-1$ whose length type is in $V_i$,
but not in any other $V_j$.
By Lemma \ref{lem:rdim},
the set of length types of solutions of rank $n-1$ in this space
cannot be covered by a finite union of $(n-2)$-dimensional spaces.
Thus one of the systems \eqref{eq:ssyst} must determine the space $V_i$.
The same holds for systems coming from all other
nonzero $2 \times 2$ minors of the matrix $(s_{ij})$,
so $E_1$ and $E_2$ have the same solutions of rank $n-1$ and length type $L$
for all $L \in V_i$ by Theorem \ref{thm:rank}.
\end{proof}

The following example illustrates the proof of Theorem \ref{thm:cover}.
It gives a pair of equations on three unknowns
where the required number of subspaces is two.
We do not know any example where more spaces would be necessary.

\begin{example}
Consider the equations
\begin{equation*}
    E_1: x_1 x_2 x_3 = x_3 x_1 x_2
    \qquad \text{and} \qquad
    E_2: x_1 x_2 x_1 x_3 x_2 x_3 = x_3 x_1 x_3 x_2 x_1 x_2
\end{equation*}
and the generalized polynomial
\begin{equation*}
\begin{split}
    s =& \sx{E_1}{x_1} \sx{E_2}{x_3}
        - \sx{E_1}{x_3} \sx{E_2}{x_1} \\
    =& X^{2 X_1 + X_2} + X^{2 X_1 + 2 X_2 + X_3} + X^{X_1 + 2 X_3}
        + X^{X_1 + X_2 + X_3} \\
    &- X^{2 X_1 + X_2 + X_3} - X^{X_1 + X_3}
        - X^{2 X_1 + 2 X_2} - X^{X_1 + X_2 + 2 X_3}.
\end{split}
\end{equation*}
If $L$ is a length type of a nontrivial solution of the pair $E_1, E_2$,
then $s(L) = 0$.
If $s(L) = 0$, then $L$ must satisfy an
equation $p = q$, where
\begin{equation*}
    p \in \{2 X_1 + X_2, X_1 + 2 X_3, X_1 + X_2 + X_3\}
    \quad \text{and} \quad
    q \in \{X_1 + X_3, 2 X_1 + 2 X_2\}.
\end{equation*}
The possible relations are
\begin{equation*}
    X_3 = 0, \qquad
    X_1 + X_2 = X_3, \qquad
    X_2 = 0, \qquad
    X_1 + 2 X_2 = 2 X_3.
\end{equation*}
If $L$ satisfies one of the first three, then $s(L) = 0$.
If $L$ satisfies the last one, then $s(L) \ne 0$, except if $L = 0$.
So if $h$ is a nonperiodic solution, then
\begin{equation*}
    |h(x_3)| = 0 \qquad \text{or} \qquad
    |h(x_1 x_2)| = |h(x_3)| \qquad \text{or} \qquad
    |h(x_2)| = 0.
\end{equation*}
There are no nonperiodic solutions with $h(x_2) = \eps$,
but every $h$ with $h(x_3) = \eps$ or $h(x_1 x_2) = h(x_3)$ is a solution.
\end{example}

\section{Independent Systems} \label{sect:indsyst}

A system of word equations $E_1, \dots, E_m$ is \emph{independent}
if it is not equivalent to any of its proper subsystems.

A sequence of nontrivial equations $E_1, \dots, E_m$ is a \emph{chain} if
\begin{equation*}
    \sol{}{}(E_1) \supsetneq \sol{}{}(E_1, E_2) \supsetneq
    \dots \supsetneq \sol{}{}(E_1, \dots, E_m) .
\end{equation*}

The question of the maximal size of an independent system is open.
The only things that are known are that
independent systems cannot be infinite \cite{AlLa85ehrenfeucht, Gu86} and
there are systems of size $\Theta(n^4)$,
where $n$ is the number of unknowns \cite{KaPl94}.
The question of the maximal size of a chain is similarly open.
For a survey on these topics, see \cite{KaSa11adian}.

An equation $u = v$ is \emph{balanced}
if $|u|_x = |v|_x$ for every unknown $x$.
Harju and Nowotka proved that
if an independent pair of equations on three unknowns
has a nonperiodic solution,
then the equations must be balanced \cite{HaNo03}.
The proof is long and it is based on a theorem of Spehner \cite{Sp86}
(or alternatively a theorem of Budkina and Markov \cite{BuMa73}),
which also has only a very complicated proof.
However, with the help of Theorem \ref{thm:cover}
we get a significantly simpler proof and a generalization
for this result.

\begin{theorem} \label{thm:balance}
Let $E_1, E_2$ be a pair of equations on $n$ unknowns
having a solution of rank $n-1$.
If $E_1$ is not balanced,
then $\sol{n-1}{}(E_1) \subseteq \sol{n-1}{}(E_2)$.
\end{theorem}

\begin{proof}
If $E_1$ is the equation $u = v$ and $h$ is a solution of $E_1$, then
\begin{equation*}
    \sum_{i=1}^n |u|_{x_i} |h(x_i)|
    = \sum_{i=1}^n |v|_{x_i} |h(x_i)|
\end{equation*}
and $|u|_{x_i} \ne |v|_{x_i}$ for at least one $i$.
Thus the set of length types of solutions of $E_1$
is covered by a single $(n-1)$-dimensional space $V$.
Because the pair $E_1, E_2$ has a solution of rank $n-1$,
$V$ is a minimal cover
for the length types of the solutions of the pair of rank $n-1$.
By Theorem \ref{thm:cover},
$E_1$ and $E_2$ have the same solutions of length type $L$ and rank $n-1$
for all $L \in V$.
\end{proof}

Another way to think of this result is that if $E_1$ is not balanced
but has a solution of rank $n-1$ that is not a solution of $E_2$,
then the pair $E_1, E_2$ causes a larger than minimal defect effect.

If $h: \Xi^* \to \Sigma^*$ is a morphism,
then the \emph{entire system} generated by $h$
is the set of all equations satisfied by $h$.
It is denoted by $K_h$.
As a consequence of Theorem \ref{thm:balance},
we get the following result about entire systems.
The case of three unknowns was proved in \cite{HaNo03}.

\begin{corollary}
If $g, h: \Xi^* \to \Sigma^*$ are morphisms of rank $n-1$ and $K_g \ne K_h$,
then $K_g \cap K_h$ contains only balanced equations.
\end{corollary}

\begin{proof}
It can be assumed that there is an equation $E_2 \in K_g \smallsetminus K_h$.
For any equation $E_1 \in K_g \cap K_h$,
$g$ is a solution of the pair $E_1, E_2$ and
$h$ is a solution of $E_1$ but not of $E_2$.
By Theorem \ref{thm:balance}, $E_1$ must be balanced.
\end{proof}

As the main application of the tools developed in this article,
the following variation of the question
about maximal sizes of chains is studied:
How long can a sequence of nontrivial equations $E_1, \dots, E_m$ be if
\begin{equation*}
    \sol{n-1}{}(E_1) \supsetneq \sol{n-1}{}(E_1, E_2) \supsetneq
    \dots \supsetneq \sol{n-1}{}(E_1, \dots, E_m) ?
\end{equation*}
We prove an upper bound
depending quadratically on the length of the first equation.
For three unknowns we get a similar bound
for the size of independent systems and chains.
Previously no bounds like those in
Theorem \ref{thm:chain} and Corollary \ref{cor:chain}
have been known.

\begin{theorem} \label{thm:chain}
Let $E_1, \dots, E_m$ be nontrivial equations on $n$ unknowns and let
\begin{equation*}
    \sol{n-1}{}(E_1) \supsetneq \sol{n-1}{}(E_1, E_2) \supsetneq
    \dots \supsetneq \sol{n-1}{}(E_1, \dots, E_m) \ne \varnothing .
\end{equation*}
If the set of length types of solutions of the pair $E_1, E_2$ of rank $n-1$
is covered by a union of $N$ $(n-1)$-dimensional subspaces,
then $m \leq N + 1$.
There are two unknowns $x, y$ such that
\begin{math}
    m \leq (|E_1|_x + |E_1|_y)^2 + 1 .
\end{math}
\end{theorem}

\begin{proof}
It can be assumed that $E_i$ is equivalent to the system $E_1, \dots, E_i$
for all $i \in \{1, \dots, m\}$.
Let the set of length types of solutions of $E_2$ of rank $n-1$
be covered by the $(n-1)$-dimensional spaces $V_1, \dots, V_N$.
Some subset of these spaces forms a minimal cover
for the length types of solutions of $E_3$ of rank $n-1$.
If this minimal cover would be the whole set,
then $E_2$ and $E_3$ would have the same solutions of rank $n-1$
by the second part of Theorem \ref{thm:cover}.
Thus the set of length types of solutions of $E_3$ of rank $n-1$
is covered by some $N-1$ of these spaces.
We conclude inductively that
the set of length types of solutions of $E_i$ of rank $n - 1$
is covered by some $N - i + 2$ of these spaces
for all $i \in \{2, \dots, m\}$.
It must be $N - m + 2 \geq 1$, so $m \leq N + 1$.
The second claim follows by Theorem \ref{thm:cover}.
\end{proof}

In the case of three unknowns, Theorem \ref{thm:chain} gives an
upper bound depending on the length of the shortest equation for the
size of an independent system of equations, or an upper bound
depending on the length of the first equation for the size of a
chain of equations. A better bound in Theorem \ref{thm:cover} would
immediately give a better bound in the following corollary.

\begin{corollary} \label{cor:chain}
If $E_1, \dots, E_m$ is an independent system on three unknowns
having a nonperiodic solution, then
\begin{math}
    m \leq (|E_1|_x + |E_1|_y)^2 + 1
\end{math}
for some $x, y \in \Xi$.
If $E_1, \dots, E_m$ is a chain of equations on three unknowns,
then
\begin{math}
    m \leq (|E_1|_x + |E_1|_y)^2 + 5
\end{math}
for some $x, y \in \Xi$.
\end{corollary}

Corollary \ref{cor:chain} means that
as soon as we take one equation on three unknowns,
we get a fixed bound
for the size of independent systems containing that equation.

It is worth noting that the bounds
in Theorem \ref{thm:chain} and Corollary \ref{cor:chain}
do not depend on the number of unknowns, only on the length of one equation.

Getting a similar bound
for the sizes of independent systems or chains
in the case of more than three unknowns
remains an open problem.
Such a bound would have to depend on the number of unknowns.
Indeed, in Theorem \ref{thm:chain} it is not enough to assume
that the equations are independent and have a common solution of rank $n - 1$.
If the number of unknowns is not fixed,
then there are arbitrarily large such systems
where the length of every equation is 10 \cite{KaPl94}.

\bibliographystyle{plain}
\bibliography{../bibtex/ref}

\end{document}